\documentclass{amsart}
\usepackage{rotating}
\usepackage{numprint}
\usepackage{graphicx}
\usepackage{siunitx}
\usepackage{booktabs}
\usepackage{xcolor}
\usepackage{comment}
\usepackage[utf8]{inputenc}
\usepackage[normalem]{ulem}
\usepackage{lscape}
\usepackage{hyperref}
\usepackage{csvsimple}
\usepackage{siunitx}
\usepackage{orcidlink}
\usepackage[colorinlistoftodos]{todonotes}

\newtheorem{theorem}{Theorem}[section]

\newtheorem{proposition}[theorem]{Proposition}
\theoremstyle{definition}

\newtheorem{conjecture}{Conjecture}[section]

\theoremstyle{remark}

\numberwithin{equation}{section}
\newcommand{\mR}{\mathbb{R}}
\newcommand{\mT}{\mathbb{T}}

\newcommand{\mZ}{\mathbb{Z}}

\newcommand{\om}{\omega}

\newcommand{\te}{\theta}

\newcommand{\rot}{\mathrm{rot}}

\begin{document}

\bibliographystyle{plain}

\title[Validated Gaps for the Almost Mathieu operator]{Computer Validation of Open Gaps for the Almost Mathieu Operator with Critical Coupling}

%    Remove any unused author tags.

%    author one information
\author{Jordi-Lluís Figueras
%\orcidlink{0000-0002-0535-4137}
}
\address{Departament de Mathematics. Uppsala Universitet. Uppsala (Sweden).}
\curraddr{}
\email{jordi-lluis.figueras@math.uu.se}
\thanks{}

%    author two information
\author{Joaquim Puig
% \orcidlink{0000-0003-0600-1023}
}
\address{Department of Mathematics. Universitat Politècnica de Catalunya. Diagonal 647, 
08028 Barcelona (Spain).}
\curraddr{}
\email{joaquim.puig@upc.edu}
\thanks{}


\begin{abstract}
We present some computer assisted methods to prove the existence of spectral gaps for the Almost Mathieu operator at critical coupling and give rigorous numerical estimates  on their size. As an example we show that the first 8 gaps predicted by the Gap Labelling theorem are open when the frequency $\omega$ is $\sqrt{5}-1)/2$ and 12 of them are open when $\omega=e-2$. A dynamical method based on the constructive conjugation to a hyperbolic cocycle and a spectral method based on the rigorous computation of the eigenvalues of finite-dimensional matrices are presented. We also present some experiments and conjectures on gap size for the associated periodic problems. 

\end{abstract}

\maketitle

\section{Introduction}

In the last decades quasi-periodic Schrödinger operators and the associated dynamical systems generated by their eigenvalue equations have been intensively studied as a valuable source of examples of different spectral and dynamical properties (see,the reviews \cite{Eliasson2009a,MARX2016a,johnson-etc:book, damanik-fillman:book}). Among these Schrödinger operators, a central focus of interest has been the Almost Mathieu Operator (AMO, see \cite{jitomirskaya2019critical} for a recent review and references therein), known also as Harper Operator, which acts on $l^2(\mZ)$ as
\begin{equation}\label{eq:AMO}
(H_{b,\omega,\theta}x)_n= x_{n+1}+x_{n-1}+ b\cos\left(2\pi{\left(\omega n+ \theta\right)}\right)x_n, \quad n \in \mZ.
\end{equation}
Here $\omega$ is called the \emph{frequency} of the quasi-periodic potential, $b$ the \emph{coupling} parameter, which is taken equal to $2$ in most of the present paper, and $\theta \in \mT=\mR/\mZ$ is a phase. 

A focus of interest has been the structure of the spectrum of the AMO in terms of $\omega$ and $b$, which is a compact subset of $\mR$. For irrational frequencies and nonzero coupling it was soon conjectured \cite{Azbel1964} and numerically observed \cite{Hofstadter1976} that the spectrum is a Cantor set: connected components of the resolvent set, the \emph{gaps}, are dense in the spectrum, see Figure \ref{fig:hofstadter}. This contrasts with the \emph{periodic case} of $\omega$ rational, where the spectrum, contrary to the case of $\omega$ irrational, depends on $\phi$ and it is the union of finitely many closed intervals, the \emph{spectral bands} (see \cite{damanik-fillman:book} for a general introduction to 1D ergodic Schrödinger operators).

The Cantor structure of the spectrum, which as a mathematical challenge was known as the \emph{Ten Martini Problem}, has been established in a series of papers \cite{bellissard-simon, Choi1990, puig:cantor, Avila2006, Avila2009} and references therein. Thanks to Aubry-André duality \cite{aubry-andre}, the spectrum for $|b|>2$ and the spectrum for $1/b$ are homothetic and the gap structure is the same. In the \emph{noncritical} case ($|b|\ne 2 $), the fact that all gaps predicted by the \emph{Gap Labelling Theorem} \cite{johnson:exponential} are open, a question called the \textit{Dry Ten Martini Problem}, has been proven for $\omega$ Diophantine \cite{Avila2010} and for all irrational frequencies \cite{avilaDryTenMartini2023}. In the \emph{critical} case $|b|=2$, although it is known that for irrational $\omega$ the spectrum is a Cantor set of zero Lebesgue measure \cite{last:zero,Avila2006}, and therefore infinitely many gaps are open, it is not known which gaps are open, even for Diophantine frequencies. The numerical plot of the spectrum versus the frequency, the well-known \emph{Hofstadter Butterfly} \cite{Hofstadter1976,papillon,Guillement1989,Osadchy2001} depicted in Figure 1 suggests an opening of all gaps for the irrational case. To our knowledge, the question of whether all possible gaps are open for the AMO remains open only in the critical case $|b|=2$ see \cite{Krasovsky2016,avilaDryTenMartini2023}. In particular, for a Diophantine $\omega$ such as the golden mean, it is not known whether a particular gap is open. 

The goal of this paper is to prove realistic bounds on the size of individual gaps of the AMO for critical coupling $b=2$ and irrational frequencies using rigorous numerics \cite{tucker:book}. The idea is the following: given a numerical approximation of a gap which can be obtained by non-rigorous means we will provide algorithms to validate it (i.e. proving that it actually represents a true gap) that reduce the problem to checking a large, but finite, number of strict inequalities. This last step can be checked rigorously by a computer using interval arithmetic. We will present two methods to validate spectral gaps: a \emph{dynamical method}, based on the dynamics of a Schrödinger cocycle inside a spectral gap, and a \emph{spectral method}, based on the validated computation of the eigenvalues of the finite dimensional matrices given by periodic approximations of the irrational frequency. 

We will show that 8 gaps of the AMO (as predicted by the Gap Labelling Theorem) are open for $b=2$ and $\omega= \frac{\sqrt{5}-1}{2}$ and 12  gaps are open for $\omega=e-2$. These specific values are chosen to exemplify the method and could be adapted to other particular values. We will also provide validated bounds for the upper and lower endpoints of these intervals, yielding an estimate of its width and an upper bound on the measure of the spectrum (which is known to be zero in this case). Finally, we investigate questions on the spectrum with rational frequencies regarding the ordering of gaps according to their length or the relation with the quantities involved appearing in Thouless conjecture.

\begin{figure}
\centering
\includegraphics[width=0.99\linewidth]{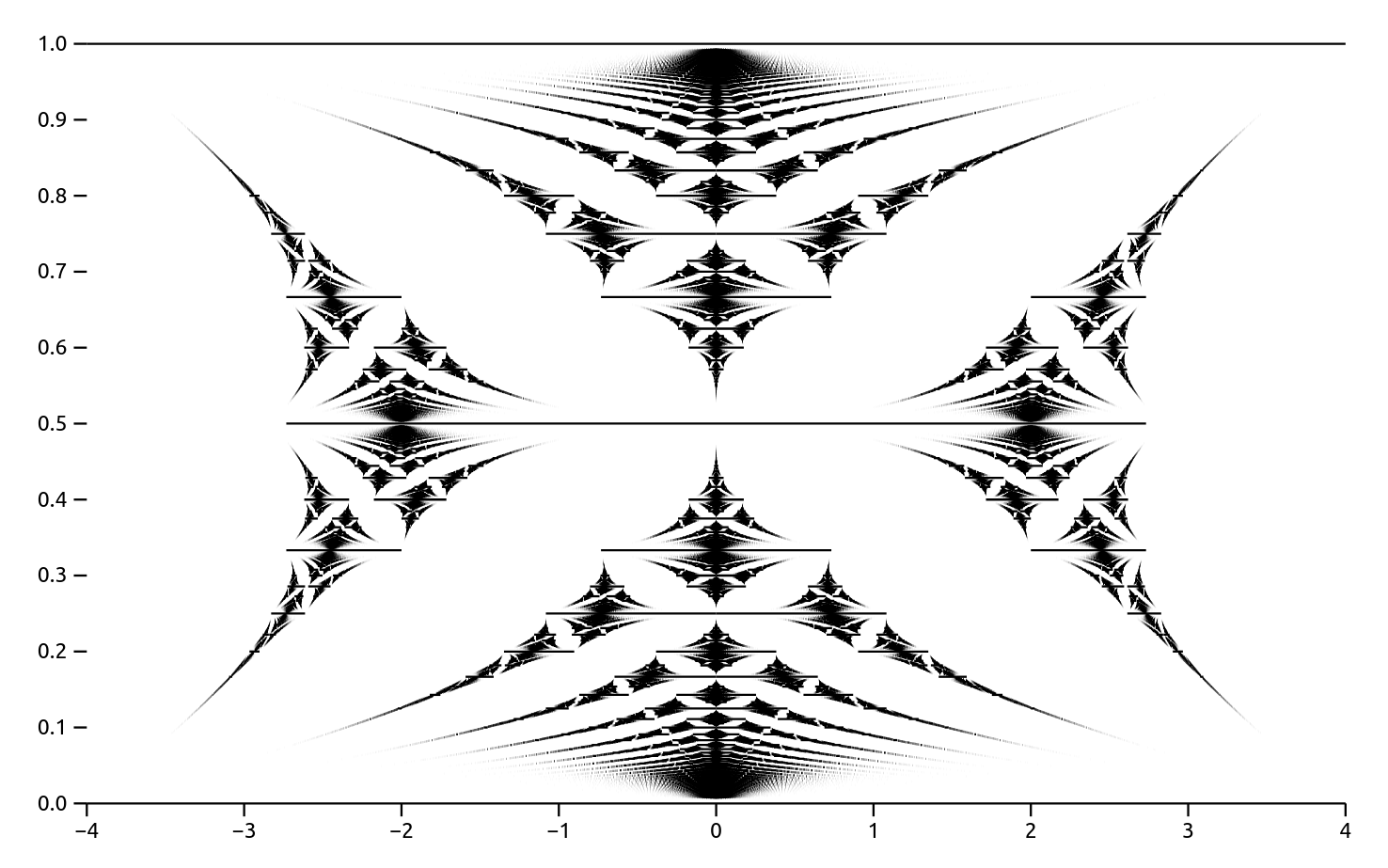}

\caption{Hofstadter butterfly: frequency on the vertical direction, spectrum on the horizontal direction.}
\label{fig:hofstadter}
\end{figure}

The paper is structured as follows. In the methods section we give the necessary background on the dynamical and spectral aspects of the problem, as well as an overview of the computer-assisted methods for the proof of the main results. In the results section we state the two main results on the existence of spectral gaps. We also describe some experiments on the size and scaling of gaps, in particular about the scaling of Thouless conjecture.   We end the paper discussing the limitations and the possible extensions of the results.

\section{Methods}

\subsection{Dynamics and Gaps of Quasi-Periodic Schrödinger operators}

In this section we summarize some of the theoretical results which will be needed in the rest of the paper. Our presentation is limited to the results which are relevant for our approach and we refer the interested reader to the bibliography given in the introduction. 

\subsubsection{Quasi-periodic cocycles and their reducibility}

In the validated numerical algorithms that we will present it is important to describe the dynamics of the eigenvalue equation (\ref{eq:AMOvaps}) or rather, its equivalent 1st order system,
\begin{equation}\label{eq:AMOskew}
\left(\begin{array}{c} x_{n+1} \\ x_{n} \end{array}\right)=
\left(
\begin{array}{cc}
a - b \cos(2\pi \theta_n)  & \;  -1 \\
1 & \; \phantom{-} 0
\end{array}
\right)
\left(\begin{array}{c} x_{n} \\ x_{n-1} \end{array}\right),\qquad
\theta_{n+1} = \theta_n + \om,
\end{equation}
when $a$ lies inside an open spectral gap.

Dynamically, first-order systems like (\ref{eq:AMOskew}), known as \emph{quasi-periodic skew-product systems}, are defined by the iteration of a map, called the \emph{cocycle},
\begin{equation}\label{eq:cocycle}
(v,\te) \in \mR^2 \times \mT \mapsto
\left(A_{a},\om\right)(v,\te)=
\left(A_{a}(\te)v,\te+\om\right) \in \mR^2 \times \mT,
\end{equation}
where
\[
A_{a}(\te)= \left( \begin{array}{cc}
a - b\cos(2\pi\te)  & -1 \\
1 & \phantom{-}0
\end{array} \right).  \]

We say that a cocycle $(A,\om)$ on $SL(2,\mR)$ is conjugated to another $(B,\omega)$, possibly not in the class of Schrödinger cocycles, if there exist continuous maps $Z,B:\mT \to SL(2,\mR)$ satisfying
\[ A(\te) Z(\te) = Z(\te + \om) B(\te), \qquad \te \in \mT. \]
A cocycle is \emph{reducible to constant coefficients} whenever a map $B$ can be found that it is independent of $\theta$, called Floquet matrix. This is a generalization of the concept of \emph{Floquet reducibility} for systems of linear   differential equations with periodic coefficients. Contrary to the periodic case, Floquet reducibility for quasi-periodic cocycles does not always hold (see, for example, \cite{Figueras_Haro_2013, karaliolios:global,you:icm2018}).

\subsubsection{The rotation number and the Gap Labelling Theorem}

The rotation number, $\rot(a,b)$, measures how non-trivial solutions of the eigenvalue equation of an AMO
\begin{equation}\label{eq:AMOvaps}
(H_{b,\omega,\theta}x)_n= x_{n+1}+x_{n-1}+ b\cos{\left(2\pi(\omega n+ \theta)\right)}x_n = a x_n, \quad n \in\mZ
\end{equation}
 wind around the origin \cite{johnson-moser,delyon-souillard}. It is related to the \emph{integrated density of states},  and can be defined for more general quasi-periodic cocycles on $SL(2,\mR)$ \cite{herman,krikorian:rigidity}. See  \cite{damanik-fillman:book} for references and a comprehensive introduction.

The rotation number for the AMO and an irrational frequency is a continuous, monotone function of $a$, which is only constant outside the spectrum of the AMO, $\Sigma_\omega$ (it is therefore independent of $\theta$ as long as $\omega$ is irrational). When $\omega=p/q$ is rational the spectrum depends on the phase $\theta$ and will be denoted as $\Sigma_{\omega}(\theta)$ with 
\begin{equation}\label{eq:periodicspectrum}
\Sigma_{\omega}= \bigcup_{\theta \in \mT} \Sigma_{\omega}(\theta).
\end{equation}

Moreover,  the open intervals in the resolvent set of the spectrum, the \emph{spectral gaps} where the rotation number is constant, are uniquely labelled by an integer called its \emph{label}. This is the content of the \emph{Gap Labelling Theorem} \cite{johnson-moser}, which asserts that inside every open gap in the resolvent set, there is a unique label $k \in \mZ$ such that the rotation number is \emph{resonant} with respect to $\omega$, that is, it is of the form $\{k\omega\}\in [0,1)$ where $\{\cdot\}$ denotes the fractional part of a real number (see Theorem 4.10.3 in \cite{damanik-fillman:book}) and the unique value $k$ is  called the \textit{label} of the gap.

Whenever the rotation number takes a resonant value for  single $a$ but it is stricly monotone at this value, we will call  the set $\{a\}$ \emph{collapsed gap}, and we will distinguish it from \emph{non-collapsed gaps},  the actual open intervals in the resolvent.

\subsubsection{Dynamics inside a spectral gap}

The dynamics of Schrödinger cocycles can display some rich behavior which is not found in their periodic counterparts (when $\omega$ is a rational number). For energies in the resolvent set (i.e. the complement of the spectrum), dynamics are well-described by \emph{uniform hyperbolicity} which has several equivalent definitions in our setting (cf. Theorem 3.8.2 in the book \cite{damanik-fillman:book}). Some of these definitions can be given in terms of the absence of nontrivial bounded solutions of the skew-product or the uniform exponential growth of the norm of the iterates of the cocycle. The concept of uniform hyperbolicity of cocycles or exponential dichotomy is very relevant in our approach: by virtue of Johnson's theorem \cite{johnson:exponential} $a$ lies in the spectrum if, and only if, the corresponding cocycle is uniformly hyperbolic (Theorem 4.9.3 in \cite{damanik-fillman:book}). 

It is therefore possible to establish the existence of an open spectral gap in a Schrödinger operator (e.g. the AMO) provided the cocycle is shown to be uniformly hyperbolic for a particular energy $a$.  It is well-known that uniform hyperbolicity is equivalent to the conjugacy to a \emph{projective} diagonal cocycle in $PSL(2,\mR)$ (see Theorem 3.8.2 in \cite{damanik-fillman:book}  and references therein). In order to prove conjugacy to a hyperbolic cocycle we will need  a  sufficient condition for uniform hyperbolicity \cite{haro-llave:numerical,Figueras_Haro_2012,Haro_Book_2016} which can be validated using rigorous numerics and is enough for our purposes.

\begin{theorem}\label{res:persistence}
Let $(A,\omega)$ be a  continuous cocycle on $SL(2,\mR)$ (for example the Schrödinger cocycle of the AMO) and a frequency $\omega \in \mR$. Assume that there exist continuous maps $P_1, P_2: \mT\to SL(2,\mR)$, two constants $\Lambda_{11},\Lambda_{22}\in\mathbb R$, satisfying ${0 < \Lambda_{11} < 1 < \Lambda_{22}}$, 
and constants $\sigma, \tau, \lambda$
such that the  following bounds hold for every $\theta \in \mT$:
\begin{itemize}
	\item[H1:] $\max_{ij}\left| (P_2(\theta+\omega) A(\theta)P_1(\theta) - \Lambda)_{ij}\right| <\sigma $, where $\Lambda= \operatorname{diag}(\Lambda_{11},\Lambda_{22})$;
	\item[H2:] $\max_{ij}\left| (P_1(\theta)P_2(\theta)-Id)_{ij}\right|< \tau$;
    \item[H3:] $\max(\Lambda_{11}, \Lambda_{22}^{-1})<\lambda$.
\end{itemize}

Then, if $\sigma+\lambda+\tau < 1$ we obtain that 
\begin{enumerate}
\item $(A,\omega)$ is uniformly hyperbolic. If $(A,\omega)$ is a Schrödinger cocycle with energy $a$, then $a$ is inside an open spectral gap.
\item There exists a continuous map $\tilde P:\mathbb T\rightarrow SL(2, \mathbb R)$ and continuous functions 
${\tilde \Lambda}_{11},{\tilde \Lambda}_{22}: \mT \to \mR$ such that 
\[
{\tilde P}(\theta+\omega)^{-1} A(\theta){\tilde P}(\theta) - \tilde\Lambda(\theta)=0, \quad
\tilde \Lambda= \operatorname{diag}(\tilde \Lambda_{11},\tilde \Lambda_{22}),
\]
The map $\tilde P$ encodes in its columns the stable and unstable directions.
\end{enumerate}
\end{theorem}

If the invariant stable and unstable bundles have non-orientable topology, we need to apply the same theorem but on the cocycle $\left(\dfrac{\omega}{2}, A_a(2\cdot)\right)$on $\mathbb{R}/(2\mathbb{Z})\times\mathbb R^2$ instead of $\mathbb{R}/\mathbb{Z}\times\mathbb R^2$.

\subsection{Dynamical Method: Validation of Uniform Hyperbolicity}

The idea for proving the existence of open gaps for Almost Mathieu operators will be to compute numerically, i.e. with float point arithmetic, sufficiently good approximations of a reducing transformation $P$ and a constant diagonal matrix $\Lambda$ so that they can be validated, using interval arithmetic, by checking hypothesis (H1-H3) in Theorem \ref{res:persistence} above. 

\subsubsection{Computation of the approximations}

The main target for this step is to get a representation of the matrix-valued map $P(\te)$ which gives the change of coordinates and the diagonal matrix $\Lambda$. Since the base dynamics are on the circle, it is natural to express them as Fourier series. There are several ways of computing such approximations for this problem, and a discussion of which one is the optimal is out of scope of this paper \cite{Figueras_Haro_Luque_2017}. In order to compute a first approximation of a fiber of the unstable (and stable) bundle of the cocycle we use a simple iteration method. Once this is computed, by simply iterating this fiber under the dynamics, we propagate it to all node points in $\mathbb T$. Finally, we use FFT to obtain the desired Fourier series.

\subsubsection{Interval arithmetic and validation of solutions}

The validation of the solutions follows the same principles as the ones developed in \cite{Haro_Book_2016} but with some improvements. The main one is that, instead of using Fourier Models for rigorously manipulating the products and operations needed in the proof, we use instead the more advanced technique of rigorous FFT developed in the paper \cite{Figueras_Haro_Luque_2017}. With this improvement we transform the computation complexity from order $N^2$ to $N\log(N)$, being $N$ the number of Fourier modes.

Another improvement is that, once we have validated the hyperbolicity of the cocycle for a spectral value $a$, we can automatically obtain an interval of $a$'s where  the cocycle is hyperbolic with the same rotation number. This is done by using a similar idea from \cite{de_la_Llave_Figueras_2017}. Assume that the inequality in hypothesis  H1 holds for a spectral value $a_0$ with upper bound $\sigma_0$:
\begin{equation}\label{eq: ineq}
\max_{ij}\left| \left(P_2(\theta+\omega) A_{a_0}(\theta) P_1(\theta) -\Lambda\right)_{ij} \right| <\sigma_0
\end{equation}
for any $\theta \in \mT$. Then, for any $\varepsilon$ we obtain that  
\begin{multline*}
P_2(\theta+\omega) A_{a_0+\varepsilon}(\theta) P_1(\theta) -\Lambda=\\
P_2(\theta+\omega) A_{a_0}(\theta) P_1(\theta) -\Lambda+
\varepsilon P_2(\theta+\omega) 
\begin{pmatrix}-1 & 0\\ 0 & 0\end{pmatrix}
P_1(\theta).
\end{multline*}

This means that the left-hand side of inequality \eqref{eq: ineq} satisfies the bound
\begin{multline*}
\max_{ij}\left| \left(P_2(\theta+\omega) A_{a_0}(\theta) P_1(\theta) -\Lambda\right)_{ij} \right|< \\
\sigma_0+\varepsilon 
\underbrace{\left(\max_{\theta}\max_{ij}\left|(P_2(\theta)_{ij}\right|\right)}_{\|P_2\|}
\underbrace{\left(\max_{\theta}\max_{ij}\left|(P_2(\theta)_{ij}\right|\right)}_{\|P_1\|}
\end{multline*}
and  any $\varepsilon$ satisfying $$|\varepsilon|< \frac{1-\sigma_0-\lambda}{\|P_2\|\|P_1\|}$$ will still fulfill the hypotheses of Theorem \ref{res:persistence}, and the cocycle will still be hyperbolic.
 
\subsection{Spectral Method: Periodic approximants}

Besides using the \emph{dynamical method}, we also compute some spectral gaps for Almost Mathieu operators using periodic approximations, which is what we call the \emph{spectral method}. This is essentially a computer-assisted validation of the original periodic approximation which Hofstadter  \cite{Hofstadter1976} used to produce his famous butterfly-shaped plots. 

\subsubsection{Persistence of gaps. Theoretical results}

This approach uses the continuity of the spectrum as a function of $\omega$ which, in a quantitative form, was established in \cite{Choi1990} and later improved in \cite{Avron1990a}. Using continuity results on the frequency $\omega$ one can prove results for \emph{irrational} values of $\omega$ if the spectrum of a sufficiently close periodic frequency is well-approximated. Indeed, let $\omega$ and $\omega'$ be two frequencies and $\Sigma_{\omega}$ and $\Sigma_{\omega'}$ be the two spectra of critical AMO. Recall that when $\omega$ is rational, $\Sigma_\omega$ is defined through the union (\ref{eq:periodicspectrum}). 

The continuity result by Choi, Elliot and Yui \cite{Choi1990} (Theorem 4.3) states that  
\begin{equation}\label{eq:CEYestimate}
\operatorname{Dist}(\Sigma_{\omega},\Sigma_{\omega'})
 \le 3 \left(2\pi |\omega'-\omega|\right)^{\frac{1}{3}},
\end{equation}
where $\operatorname{Dist}(K_1, K_2)$, for any two compact subsets $K_1$ and $K_2$, refers to the Hausdorff distance between $K_1$ and $K_2$. The above formula on the distance between the two spectra as sets can be written in terms of gaps. Indeed, once an open gap exists for a certain value frequency $\omega$, 

\begin{proposition}[cf \cite{Choi1990}]\label{res:estimatesonpersistenceCEY} Let $k\in \mZ$ and let $(a_k^-(\omega),a_k^+(\omega))$ be an open spectral gap with label $k$ in $\Sigma_{\omega}$ and $\gamma_k(\omega)=a_k^+(\omega)-a_k^-(\omega)$ its length.  If $\omega' \in [0,1)$ satisfies that 
\[
\gamma_k(\omega)> 6 \left(2\pi |\omega'-\omega|\right)^{\frac{1}{3}}
\]
then the gap with label $k$ is open for $\omega'$ and it includes the interval
\[
\left( 
a_k^-(\omega)+3 \left(2\pi |\omega'-\omega|\right)^{\frac{1}{3}}
,
a_k^+(\omega)-3 \left(2\pi |\omega'-\omega|\right)^{\frac{1}{3}}
\right)
\subset
(a_k^-(\omega'),a_k^+(\omega')).
\]
and, therefore, its length satisfies \begin{equation}\label{eq:estimateongaplengthCEY}
\gamma_k(\omega')>\gamma_k(\omega)-6 \left(2\pi |\omega'-\omega|\right)^{\frac{1}{3}}.
\end{equation}
\end{proposition}

\begin{proof}
Using inequality \ref{eq:CEYestimate} one has that, for any $a\in \mR$,
\[
d(a,\Sigma_\omega)\le d(a,\Sigma_{\omega'})+d(\Sigma_\omega,\Sigma_{\omega'})
\le d(a,\Sigma_{\omega'})+3 \left(2\pi |\omega'-\omega|\right)^{\frac{1}{3}}.
\]
Under the present hypothesis, if 
\[
a \in
\left( 
a_k^-(\omega)+3 \left(2\pi |\omega'-\omega|\right)^{\frac{1}{3}}
,
a_k^+(\omega)-3 \left(2\pi |\omega'-\omega|\right)^{\frac{1}{3}}
\right)
\] 
then $d(a,\Sigma_\omega)> 3 \left(2\pi |\omega'-\omega|\right)^{\frac{1}{3}}$ and also $d(a,\Sigma_{\omega'})>0$ which, given that $\Sigma_{\omega'}$ is compact, implies that $a$ lies inside a gap of $\Sigma_{\omega'}$. The latter holds for \emph{every} $\omega''$ between $\omega$ and $\omega'$ so, by continuity of the label \cite{Choi1990}, if $a$ belongs to the gap with label $k$ for $\omega$, then it must also belong to the gap labelled by the same $k$ for $\omega'$.  
\end{proof}

In particular, if a certain computer method proves rigorously the existence of a certain gap of length $l$ for a particular value of $\omega$, this gap will persist for nearby values of $\omega'$ through an explicit formula which involves $l$ and the difference $|\omega'-\omega|$:
\begin{equation}
%\label{eq:persistencechoi}
|\omega'-\omega| < \frac{l^3}{6^3 2\pi} = \frac{l^3}{432\pi} < \frac{l^3}{1357}
\end{equation}

This bound was  improved by Avron, van Mouche and Simon \cite{Avron1990a} (Theorem 7.3) to an exponent of $1/2$. As Krasovsky notes in \cite{Krasovsky2016} (Section 4), if $\alpha$ belongs to $\Sigma_{\omega}$, there exists $\alpha'\in \Sigma_{\omega'}$ such that
\begin{equation}\label{eq:HolderK}
\left|\alpha-\alpha'\right| \le 60 \left|\omega-\omega'\right|^{1/2}.
\end{equation}
and the constant $60$ can be replaced by a better bound which depends on the difference $\omega-\omega'$. Similarly to Proposition \ref{res:estimatesonpersistenceCEY} we can now summarize the implications that these bounds have on the estimates of gap size:

\begin{proposition}[cf \cite{Avron1990a}]\label{res:estimatesonpersistenceASM} Let $k\in \mZ$ and let $(a_k^-(\omega),a_k^+(\omega))$ be an open spectral gap with label $k$ in $\Sigma_{\omega}$ and $\gamma_k(\omega)=a_k^+(\omega)-a_k^-(\omega)$ its length.  If $\omega' \in [0,1)$ satisfies that 
\[
\gamma_k(\omega)> 120 |\omega-\omega'|^{1/2}
\]
then the gap with label $k$ is open for $\omega'$ and it includes the interval
\[
\left( 
a_k^-(\omega)+60 |\omega-\omega'|^{1/2}
,
a_k^+(\omega)-60 |\omega-\omega'|^{1/2}
\right)
\subset
(a_k^-(\omega'),a_k^+(\omega'))
\]
and therefore its length satisfies
\begin{equation}\label{eq:estimateongaplength}
\gamma_k(\omega')>\gamma_k(\omega)-120|\omega-\omega'|^{\frac{1}{2}}.
\end{equation}
\end{proposition} 

Although for small values of gap length $l$ the estimates of Proposition \ref{res:estimatesonpersistenceASM} give wider intervals, for intermediate values of $l$ Proposition \ref{res:estimatesonpersistenceCEY} may give wider approximations, so it is convenient to compute both of them. 

\subsubsection{Almost Mathieu operators with rational frequencies}

The spectrum of Almost Mathieu operators $H_{\omega,\theta}$ with a rational frequency $\omega=\frac{p}{q}$,  $p$ and $q$ coprime, is the union of a finite number of spectral bands separated by spectral gaps, all depending on $\theta$. Indeed, the spectrum of $H_{b,p/q,\theta}$ is given by the endpoints of the possible spectral gaps and bands:
\[
E_0 > E_1 \ge E_2 > E_3 \ge E_4 > \dots E_{2j-1}\ge E_{2j}>E_{2j+1} \ge \dots >E_{2q-1}
\]
so that the spectrum is given by the union of the bands
\[
\Sigma_{p/q}(\theta)= 
[E_{2q-1}, E_{2q-2}] \cup \dots \cup [E_1,E_0]=
\bigcup_{j=0}^{q-1} [E_{2j+1},E_{2j}].
\]

Spectral bands are separated by $q-1$ (possibly void) spectral gaps
\[
\Delta_k(\theta)= (E_{2q-2k},E_{2q-2k-1}), \quad k=1,\dots,q-1.
\]
Although spectral gaps will usually depend on $\theta$ for these rational frequencies, the integrated density of states is always $k/q$ on the $k$-th gap \cite{bellissard-simon}. 

Each of the endpoints of spectral bands is associated to the periodic or antiperiodic problem for a $q$-dimensional matrix. Indeed, for $k\equiv 0,3$ (mod 4), the endpoints $E_k$ correspond to the eigenvalues of the matrix 
\begin{equation}\label{eq:periodicmatrix}
H^+_{p/q}(\theta)=
\begin{pmatrix}
2\cos\left({2\pi}\theta\right) &  &  & 1 \\ 
 & 2\cos{2\pi}\left(\theta+\frac{p}{q}\right) &  &  \\ 
 &  & \ddots &  \\ 
1 &  &  & 2\cos{2\pi}\left(\theta+(q-1)\frac{p}{q}\right)
\end{pmatrix}
\end{equation}
and, for $k\equiv 1,2$ (mod 4), these correspond to the eigenvalues of the matrix $H^-_{p/q}(\theta)$ with anti-periodic boundary condition:
\begin{equation}\label{eq:antiperiodicmatrix}
H^-_{p/q}(\theta)=
\begin{pmatrix}
2\cos\left({2\pi}\theta\right) &  &  & -1 \\ 
 & 2\cos{2\pi}\left(\theta+\frac{p}{q}\right) &  &  \\ 
 &  & \ddots &  \\ 
-1 &  &  & 2\cos{2\pi}\left(\theta+(q-1)\frac{p}{q}\right)
\end{pmatrix}. 
\end{equation} 
Equivalently, these eigenvalues can also be found as the values of $a$ where the trace of the $q$-th iterate of the transfer matrix
\begin{equation}\label{eq:tracetransfer}
\begin{pmatrix}
	a - 2\cos{2\pi}\left(\theta+(q-1)\frac{p}{q}\right) & -1 \\
	1 & 0
\end{pmatrix} 
\dots
\begin{pmatrix}
	a - 2\cos{2\pi}\theta & -1 \\
	1 & 0
\end{pmatrix} 
\end{equation}
is either $+2$ for periodic eigenvalues or $-2$ for anti-periodic eigenvalues.

Almost Mathieu operators with rational frequencies are very unique because the union of the spectra over all phases $\theta$,  which defines the spectrum $\Sigma_\omega$ in Equation (\ref{eq:periodicspectrum}), can be restricted to \emph{two} particular values of $\theta$ as a consequence of Chambers Formula \cite{bellissard-simon}.  This is the contents of the following theorem by Bellissard and Simon \cite{bellissard-simon} which states which are the phases $\theta$ that give the intersection of all gaps and thus the gaps of the spectrum $\Sigma_\omega$ as the union of all spectra $\Sigma_\omega(\theta)$ for a periodic frequency $\omega=p/q$.

\begin{theorem}[cf \cite{bellissard-simon}]\label{res:periodicspectrum}
Consider an Almost Mathieu operator with coupling $|b|=2$, rational frequency $\omega=p/q$ and some $\theta \in \mathbb{T}$. Let $\Delta_k(\theta)$ be the spectral gaps of the spectrum $\Sigma_\omega(\theta)$ for $k=0,\dots,q-1$. The gaps satisfy
\begin{eqnarray*}
\Delta_k(0)  = \bigcap_{\theta}\Delta_k(\theta), \quad k=1,3,\ldots \\	
\Delta_k\left(\frac{1}{2q}\right)  =  \bigcap_{\theta}\Delta_k(\theta), \quad k=2,4,\ldots 
\end{eqnarray*}
\end{theorem}

This result reduces the computation of the spectrum of AMO with rational frequencies to the computation of the eigenvalues of two self-adjoint matrices of dimension $q$,  namely (\ref{eq:periodicmatrix}) and (\ref{eq:antiperiodicmatrix}). Note that this is precisely the justification for the algorithms for plotting  Hofstadter butterfly in Figure \ref{fig:hofstadter} (see also \cite{Guillement1989,Osadchy2001} for some refinements to reduce the dimension and the recent software \cite{HofstadterTools}). Note that this result is specific of the Almost Mathieu operator and does not necessarily hold if we replace the cosine by another periodic function. Using the continuity of the spectrum given by either Proposition \ref{res:estimatesonpersistenceCEY} or \ref{res:estimatesonpersistenceASM}, each of the gaps obtained for the periodic problem persists for an open interval of frequencies depending on gap length. Of course, this procedure can only yield opening of all gaps (the {\emph{Dry Ten Martini Problem}} \cite{avilaDryTenMartini2023}, a question which is out of the scope of this problem) when lower estimates on gap length are available for any $q$.

\subsubsection{Computer-assisted proof of the persistence of gaps for irrational frequencies}

With the previous background in mind, we can now describe the algorithm to prove the existence of gaps for an irrational frequency $\omega$ from the periodic problem. 

We take a certain rational approximation $p/q$ of $\omega$, with $p$ and $q$ coprime, and compute a validated approximation of the eigenvalues of the $q$-dimensional matrices $H^+_{p/q}(\theta)$ and $H^-_{p/q}(\theta)$ for $\theta=0$ and $\theta=1/2q$. In our case we used the Arblib library \cite{Johansson2017arb}, a C library for interval arithmetic, which provides the verified computation of eigenvalues for a matrix in interval arithmetic, see \cite{rump2010verification, vanderhoeven:hal-01579079}. We chose to compute only rational frequencies $\omega=p/q$ with $q$ being an odd number to avoid the unnecessary complication of having $0$ as a double eigenvalue (see Figure \ref{fig:hofstadter}).

Using Theorem \ref{res:periodicspectrum}, we construct the spectral bands and gaps of the periodic problem $\Sigma_{p/q}$, representing each endpoint of bands by an interval with rigorously contains the true endpoint. Additionally, this gives rigorous upper and lower bounds on the measure of the spectrum for this rational value and of every gap length. 

We can use the bounds given by Proposition \ref{res:estimatesonpersistenceCEY} or \ref{res:estimatesonpersistenceASM} to give, for every gap in the periodic problem with rational frequency $p/q$ that we have taken, a rigorous bound for which this gap will exist for any $\omega$ close enough to $p/q$ as well as a rigorous lower bound on its size. Indeed, taking an irrational $\omega$ (like $(\sqrt{5}-1)/2$ or $e-2$) and some rational approximation $p/q$, we can prove the existence of all the gaps for irrational $\omega$ which are larger than a certain bound (given by the above two continuity results on the spectrum) and depends only on the distance between $\omega$ and this rational approximation $p/q$.

\section{Results}

\subsection{Dynamical Method}

\subsubsection{Results for $\omega_1=(\sqrt{5}-1)/2$}

Using the dynamical method we obtained the gaps displayed in Table \ref{tab:golden_qp}. A total of 8 gaps were obtained, a part from the two unbounded spectral gaps at the bottom and the top of the spectrum. 
The distance between the upper and the lower point of the spectrum minus the total sum of validated gaps for $\omega_1$  was $0.635$ which gives an upper bound of the measure of the spectrum (which is known to be zero) and represents the $86.63\%$ of the interval between the validated right endpoint of the lowest gap and the right endpoint of the upper gap. The computation took 4h00m at a i7 single processor @1.8GHz.

\begin{table}
\begin{tabular}{llr}
\toprule
   Left Endpoint &    Right Endpoint &  Label \\
\midrule
       $-\infty$ &    -2.59851518536 &        0 \\
  -2.50857481536 &    -2.37612614288 &     -3 \\
  -2.30758848800 &    -2.00845718154 &     2 \\
  -1.789961146545 &   -0.273318241882 &   -1 \\
 -0.1298164588928 &  -0.0196121757507 &    4 \\
 0.0196121757507 &   0.1298164588928 &      -4 \\
  0.273318241882 &    1.789961146545 &      1 \\
   2.00845718154 &     2.30758848800 &     -2 \\
   2.37612614288 &     2.50857481536 &     3 \\
   2.59851518536 &         $+\infty$ &     0 \\
\bottomrule
\end{tabular}
\caption{Gaps for $\omega_1=(\sqrt{5}-1)/2$ obtained from the dynamical  method using double precision interval arithmetics and up to 1024 Fourier modes. The values printed in the left and right endpoints correspond to rounding values
guaranteeing they are inside the validated spectral gaps. The last column of the table is the label of the original gap in the periodic problem.}\label{tab:golden_qp}
\end{table}

\subsubsection{Results for $\omega_2=e-2$}

With the same dynamical method we obtained the gaps displayed in Table \ref{tab:golden_e-2}. A total of 8 gaps were obtained. 
The distance between the upper and the lower point of the spectrum minus the total sum of validated gaps for $\omega_2$  was $0.8273$ which gives an upper bound of the measure of the spectrum (which is known to be zero) and represents the $84.3\%$ of of the interval between the validated right endpoint of the lowest gap and the right endpoint of the upper gap. The computation took 4h00m at a i7 single processor @1.8GHz.

\begin{table}
\begin{tabular}{llr}
\toprule
    Left Endpoint &    Right Endpoint & Label \\
\midrule
        $-\infty$ &    -2.72603125000 &     0 \\
   -2.65714062500 &    -2.63451562500 &     -4 \\
   -2.60421010408 &    -2.61414062500 &    3 \\
   -2.48323437500 &   -0.801437500000 &    -1 \\
  -0.589406250000 &   -0.096787500000 &       2 \\
   0.096787500000 &    0.589406250000 &      -2 \\
   0.801437500000 &     2.48323437500 &       1 \\
    2.55656250000 &     2.61414062500 &      -3 \\
    2.63451562500 &     2.65714062500 &      4 \\
    2.72603125000 &         $+\infty$ &       0 \\
\bottomrule
\end{tabular}
\caption{Gaps for $\omega_1=e-2$ obtained from the dynamical method using double precision interval arithmetics and up to 1024 Fourier modes. The values printed in the left and right endpoints correspond to rounding values
guaranteeing they are inside the validated spectral gaps. The last column of the table is the label of the original gap in the periodic problem.}\label{tab:golden_e-2}
\end{table}

\subsection{Spectral Method}

\subsubsection{Results for $\omega_1=(\sqrt{5}-1)/2$}

%Here we insert the computed values
% Results for \omega = (\sqrt{5}-1)/2
\def\goldenapproximant{987/1597}
\def\goldenHpHqHdps{50}
\def\enlargeHgolden{0.0251248715580 \pm 4.46 \cdot 10^{-14}}
% Measure of the spectrum of the periodic problem. 
\def\measureHspectrumHpHqHgolden{0.00584216754614 \pm 3.31 \cdot 10^{-15}}
\def\goldenmeanopenspectralgaps{8}
% Measure of the intervals proved to be in the resolvent set
\def\measureHresolventHpHqHgolden{4.17228197900 \pm 3.49 \cdot 10^{-12}}
% Percentage of proved intervals
\def\percentHresolventHpHqHgolden{83.17 \pm 6.42 \cdot 10^{-4}}
\def\boundHmeasureHgolden{0.844250239140 \pm 2.13 \cdot 10^{-13}}
\def\goldenCEYestimate{1.01090201763 \pm 3.96 \cdot 10^{-12}}
\def\goldenNUestimate{2.59144268092 \pm 3.92 \cdot 10^{-12}}
\def\goldenCEYconjecture{1.89689263385 \pm 2.93 \cdot 10^{-12}}
\def\goldenSmallestGap{3.01735109613\cdot 10^{-8}}
% Results for \omega=e-2
\def\eminustwoapproximant{719/1001}
\def\eminustwoHpHqHdps{40}
\def\enlargeHe{0.0199157821095 \pm 7.79 \cdot 10^{-16}}
% Measure of the spectrum of the periodic problem. 
\def\measureHspectrumHpHqHe{0.00932062542567 \pm 1.04 \cdot 10^{-15}}
% Measure of the intervals proved to be in the resolvent set
\def\boundHmeasureHe{4.39897520877 \pm 3.90 \cdot 10^{-12}}
% Number of validated gaps
\def\eminustwoopenspectralgaps{12}
\def\eminustwoSmallestGap{1.51034220345\cdot 10^{-13}}
\def\eminustwoCEYestimate{1.02993096753 \pm 8.70 \cdot 10^{-13}}
\def\eminustwoNUestimate{4.75030275304 \pm 4.26 \cdot 10^{-12}}
\def\eminustwoCEYconjecture{2.34494604133 \pm 4.46 \cdot 10^{-12}}

%%%%%%%%%%%%%%%%%%%%%%%
\def\QmaxThouless{95}
\def\ThoulessHdps{500}
\def\ThoulessTotal{1850}
\def\CEYeightpq{3.05124276000 \pm 1.56 \cdot 10^{-12}}
\def\SmallestGappq{9.43425512814\cdot 10^{-47}}
\def\GapLengthConjecture{true}
\def\ThoulessScalingConjecture{true}
\def\LowerBoundHolderConstant{7.03127155516 \pm 3.80 \cdot 10^{-12}}
%%%%%%%%%%%%%%%%%%%%%%%%%%%%%%%%%%%%%%%%%%%%%%%%%%%%%%%%%%%%%%%
%%%%%%%%%%%%%%%%%%%%%%%%%%%%%%%%%%%%%%%%%%%%%%%%%%%%%%%%%%%%%%%
%%%DYNAMICAL%%%%%%%%%%%%%%%%%%%%%%%%%%%%%%%%%%%%%%%%%%%%%%%%%%
\def\goldengaps{8}
\def\eminustwogaps{12}
%%%%COnclusion%%%%%%
\def\goldenmeanopenspectralgaps{8}
\def\eminustwoopenspectralgaps{12}

We proved the existence of  gaps for $\omega_1=(\sqrt{5}-1)/2$  which are given in Table \ref{tab:golden_periodic}. We used $p/q= \goldenapproximant$ and ball arithmetic with $\goldenHpHqHdps$ digits and showed the existence of $\goldenmeanopenspectralgaps$ spectral gaps (apart from bounds of the lowest and upper spectral gap). In this case, we used the estimates from Proposition \ref{res:estimatesonpersistenceASM} to obtain that gaps boundaries had to be enlarged by  $\enlargeHgolden$ . The measure of the spectrum for the rational approximation $\Sigma(p/q)$ was   $\measureHspectrumHpHqHgolden$. The distance between the upper and the lower point of the spectrum minus the total sum of validated gaps for $\omega_1$  was $\boundHmeasureHgolden$  which gives an upper bound of the measure of the spectrum (which is known to be zero) and represents the $\percentHresolventHpHqHgolden \%$ of the interval between the validated right endpoint of the lowest gap and the right endpoint of the upper gap. The computation took 3h50m at a i7 single processor @1.8GHz.

\begin{table}[ht]
\begin{small}
\begin{tabular}{llrr}
\toprule
   Left Endpoint &    Right Endpoint &  Number of Gap &  Label \\
\midrule
       $-\infty$ &    -2.62263996148 &              0 &      0 \\
  -2.49080746353 &    -2.39389225666 &            233 &     -3 \\
  -2.29908225829 &    -2.01696396030 &            377 &      2 \\
  -1.84909334969 &   -0.214184658337 &            610 &     -1 \\
 -0.110814328134 &  -0.0386155348422 &            754 &      4 \\
 0.0386155348422 &    0.110814328134 &            843 &     -4 \\
  0.214184658337 &     1.84909334969 &            987 &      1 \\
   2.01696396030 &     2.29908225829 &           1220 &     -2 \\
   2.39389225666 &     2.49080746353 &           1364 &      3 \\
   2.62263996148 &         $+\infty$ &           1597 &      0 \\
\bottomrule
\end{tabular}

\caption{Validated Gaps for $\omega_1=(\sqrt{5}-1)/2$ obtained from the periodic approximation $p/q=\goldenapproximant$ using $\goldenHpHqHdps$ digits in ball arithmetic. The printing precision is set to a lower value for readability. The values printed in the left and right endpoints correspond to exact floating-point numbers which been rounded upward or downward to guarantee that they are inside the validated spectral gaps. The third column displays the ordering of the gap (from left to right) in the periodic problem: every validated gap comes from an gap in the periodic approximation $p/q$. The last column of the table is the label of the original gap in the periodic problem.}\label{tab:golden_periodic}
\end{small}
\end{table}

\subsubsection{Results for $\omega_2=e-2$}

Using the same method, we also computed gaps in the resolvent set using interval arithmetic with $\eminustwoHpHqHdps$ digits and $p/q = \eminustwoapproximant$, which is a rational approximation of $\omega_2=e-2$. Using again the estimates from Proposition \ref{res:estimatesonpersistenceASM} we obtained an enlargement value of $\enlargeHe$ and this proved the existence of the $\eminustwoopenspectralgaps$ gaps described in Table \ref{tab:e_periodic}. We note that, contrary to the golden mean case, it was enough to consider less digits to obtain more gaps. The computation took 1h40m at a i7 single processor @1.8GHz.

\begin{table}[ht]
\begin{small}
\begin{tabular}{llrr}
\toprule
    Left Endpoint &    Right Endpoint &  Number of Gap &  Label \\
\midrule
        $-\infty$ &    -2.74495255625 &              0 &      0 \\
   -2.65721095492 &    -2.64443590307 &            127 &     -4 \\
   -2.60421010408 &    -2.56648816182 &            155 &      3 \\
   -2.46431695653 &   -0.820368049989 &            282 &     -1 \\
  -0.646689943693 &   -0.611187586183 &            409 &     -5 \\
  -0.570487209785 &   -0.116609296737 &            437 &      2 \\
 -0.0740258280376 &  -0.0583643948664 &            465 &      9 \\
  0.0583643948664 &   0.0740258280376 &            536 &     -9 \\
   0.116609296737 &    0.570487209785 &            564 &     -2 \\
   0.611187586183 &    0.646689943693 &            592 &      5 \\
   0.820368049989 &     2.46431695653 &            719 &      1 \\
    2.56648816182 &     2.60421010408 &            846 &     -3 \\
    2.64443590307 &     2.65721095492 &            874 &      4 \\
    2.74495255625 &         $+\infty$ &           1001 &      0 \\
\bottomrule
\end{tabular}
\caption{Same as in Table \ref{tab:golden_periodic} but with $\omega_2=e-2$, $p/q=\eminustwoapproximant$ and using $\eminustwoHpHqHdps$ digits in ball arithmetic.}\label{tab:e_periodic}
\end{small}
\end{table}

\subsubsection{Remarks on the gaps of $\Sigma_{p/q}$}

The use of validated numerics with large precision allows a detailed analysis of the size of gaps for the periodic approximations of $\omega_1=\goldenapproximant$ and $\omega_2=\eminustwoapproximant$. 

Figure \ref{fig:goldengapdecaysemilogy} displays the general decay of gap length according to the label of each gap. Gaps tend to decay according to the label in general terms, but this decay is not strict, as it can be observed from the oscillating behaviour seen in the figure and can be checked for lower values of $q$.  For instance, for $p=13$ and $q=21$, the gap with label $-5$ has length $0.012767645482867 \pm 9.83\cdot 10^{-16}$ and the gap with label $-6$ has length $0.02806352869759 \pm 4.35\cdot 10^{-15}$. 

\begin{figure}
	\centering
	\includegraphics[width=0.99\linewidth]{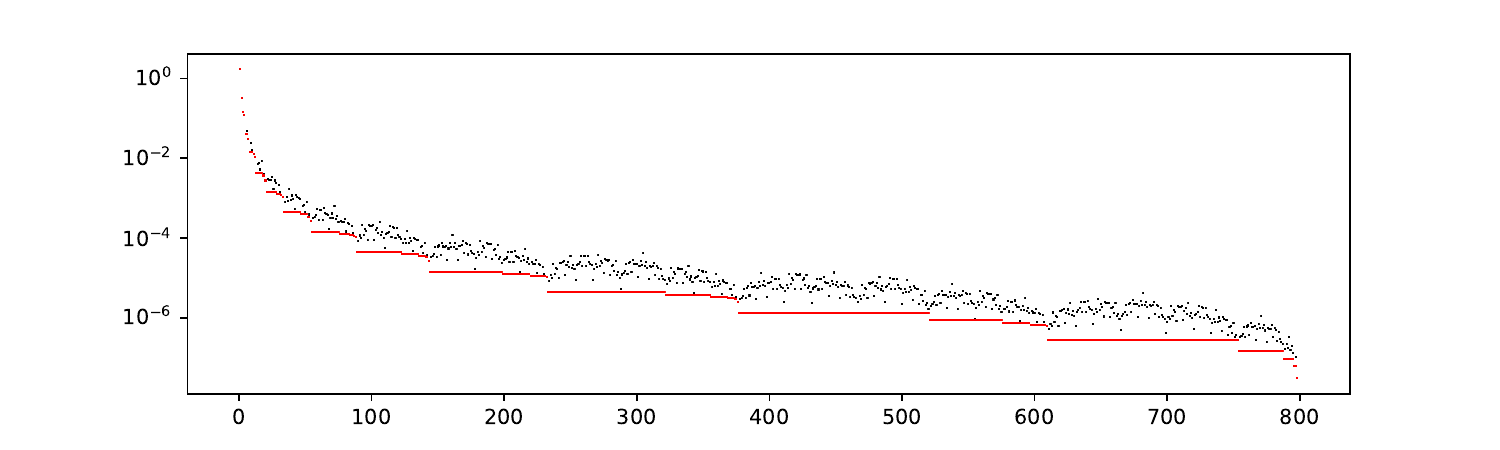} \\
	\includegraphics[width=0.99\linewidth]{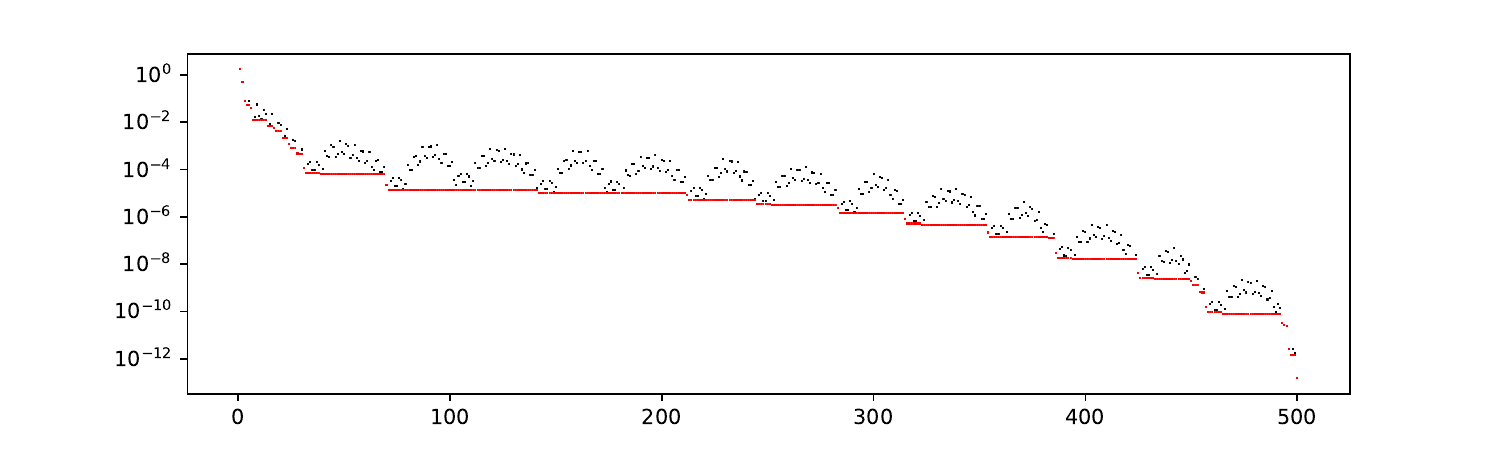}
	\caption{Length of gaps for $p/q=\goldenapproximant$ (top) and $p/q = \eminustwoapproximant$ (bottom) as a function of the absolute value of its label, $|k|$. In red, we plot the cummulative minima of gap length as a function of the absolute value of the label. Note the logarithmic scale on the vertical axis.}
	\label{fig:goldengapdecaysemilogy}
\end{figure}

The consecutive relative minima of gap length as a function of $|k|$ (the lower peaks of Figure \ref{fig:goldengapdecaysemilogy}) are important regarding the open question of whether all gaps are open for the $\omega_1$ and $\omega_2$. These minima appear at some particular values of $k$ which are related to the best rational approximations of $\omega_1$ and $\omega_2$. For $\omega_1=(\sqrt{5}-1)/2$ these consecutive minima of gap length occur precisely when the labels $k$ take values in the  Fibonacci sequence ($|k|=1,2,3,5,\dots$). Something similar happens for $\omega_2=e-2$, although in this case the rational best approximations are $1, 3, 4, 7, 32, 39, 71, 465, 536, \dots$.

We can observe that  gap lengths seem to decay faster for $\omega_2$ than for $\omega_1$. In fact, we checked that, in both cases, the following scaling law holds
\[
\gamma_{k} > \frac{1}{|k|^\nu}, \quad |k|>1,
\]
with $\nu_1=\goldenNUestimate$ for $p/q=\goldenapproximant$ and $\nu_2=\eminustwoNUestimate$ for $p/q= \eminustwoapproximant$.

Lower bounds on the gap length are important in the proofs of gap opening, but they seem to depend very much on the arithmetical properties of the frequency and the label of each gap. In \cite{Choi1990}  (Theorem 3.3) it was shown that the gaps for the periodic problem are greater than $8^{-q}$  while in \cite{Avila2009} (Theorem 7.2), the following was proved: for any irrational frequency and sequence of $p_n/q_n$ converging to it, there is a $\epsilon>0$ such that gaps of the periodic problem $p_n/q_n$ have size at least $e^{-\epsilon q_n}$.  For the golden mean $\omega_1$ and $p/q = \goldenapproximant$ we checked that all gaps are greater than $K_1^{-q}$ for $K_1=\goldenCEYestimate$, while for $\omega_2=e-2$ and $\eminustwoapproximant$ this value was $K_2=\eminustwoCEYestimate$.

\subsubsection{Exploration of different rational frequencies}

To see if similar estimates may hold for other frequencies, we computed gaps and bands for all rational frequencies for $p/q$ with odd $q$, $p$ and $q$ coprime and $q\le \QmaxThouless$. We used this to generate a database of $\ThoulessTotal$ values of $p/q$ and check different properties on gap decay. The total set of gaps which was computed is shown in Figure \ref{fig:validated-gaps}.

\begin{figure}
    \centering
    \includegraphics[width=1\linewidth]{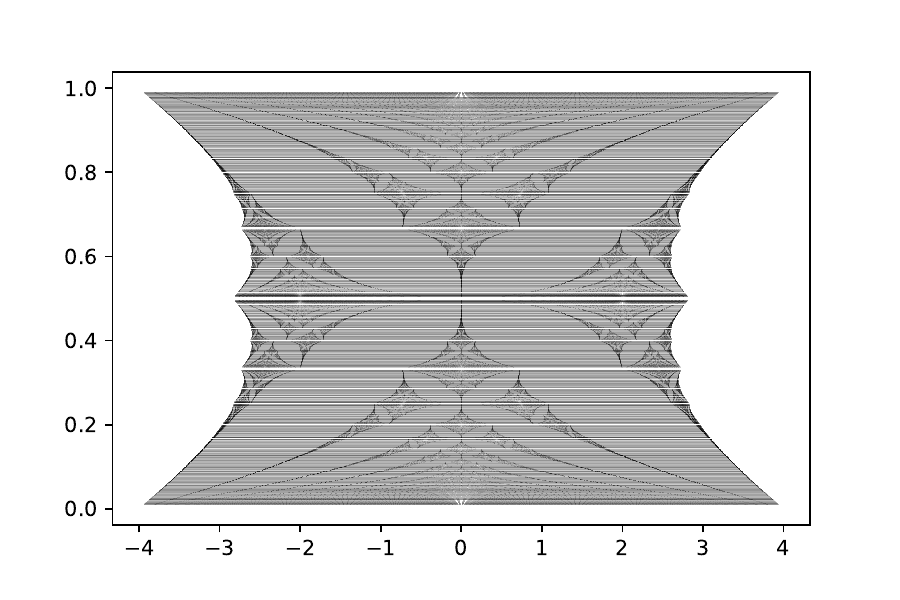}
    \caption{Validated gaps (horizontal direction)  in $\Sigma_{p/q}$ as a function of rational frequencies  $p/q$ (vertical direction) with odd $q$, $p$ and $q$ coprime and $q\le \QmaxThouless$.}
    \label{fig:validated-gaps}
\end{figure}

The gap with minimal length that was computed was $\SmallestGappq$ and this is why we used $\ThoulessHdps$ digits of internal precision, as opposed to the usual \texttt{float} precision. We checked that every gap was larger than $C^{-q}$ where $C=\CEYeightpq$, smaller than the bound $8$ obtained in \cite{Choi1990}. 

\begin{figure}
    \centering
    \includegraphics[width=0.99\linewidth]{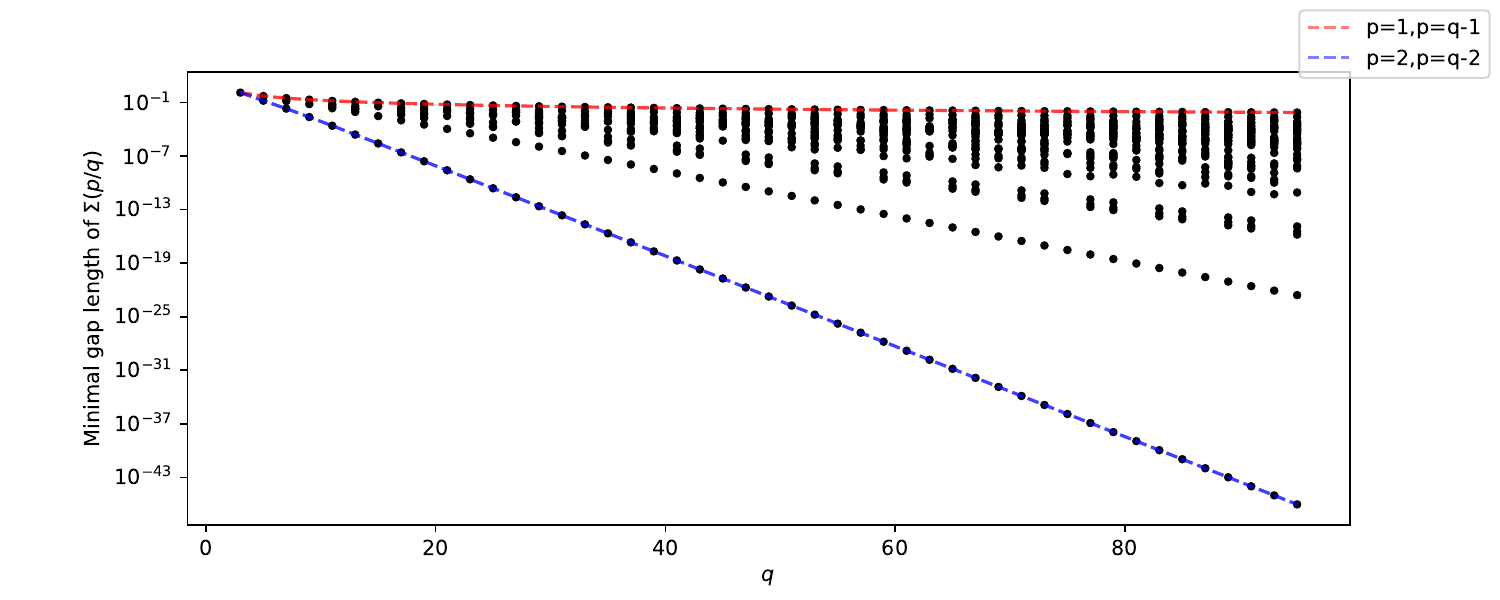}
    \includegraphics[width=0.99\linewidth]{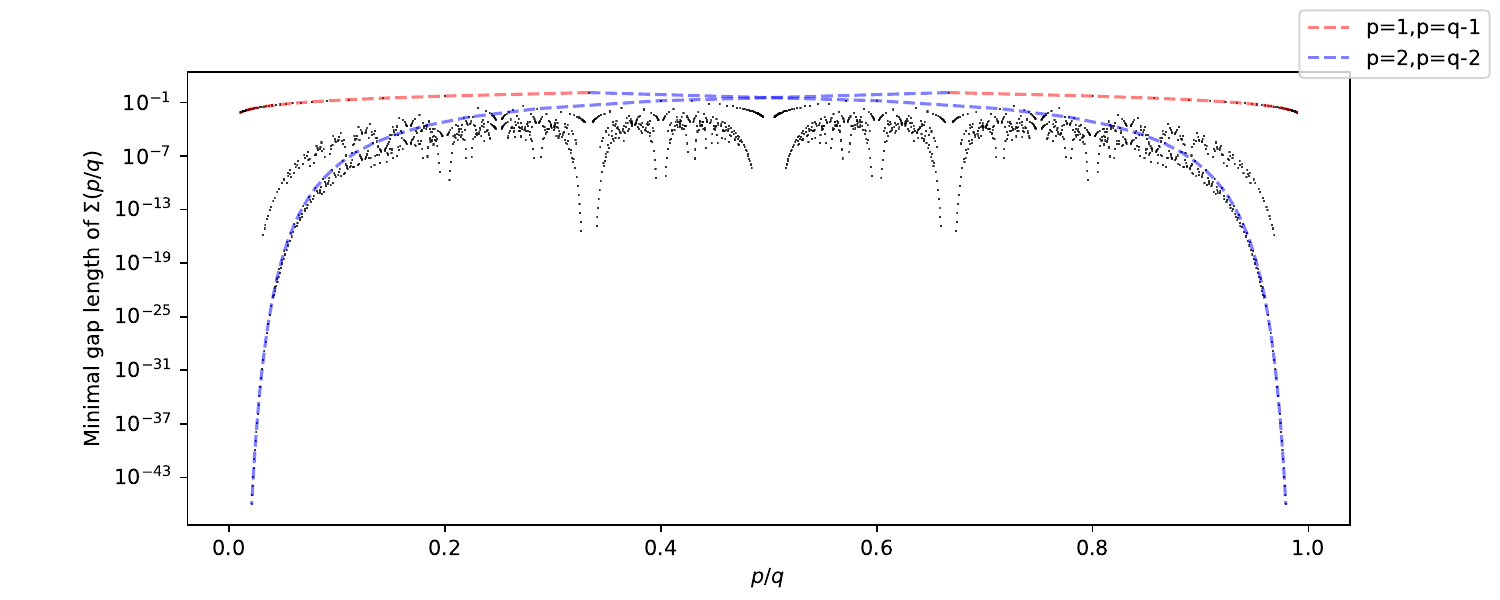}
    \caption{Top: Minimal Gap length for all rational frequencies for $p/q$ with odd $q$, $p$ and $q$ coprime and $q\le \QmaxThouless$. Note the logarithmic scale in the vertical axis. Bottom: the same plot as a function of $p/q$.}
    \label{fig:pq_gap_decay}
\end{figure}

Figure \ref{fig:pq_gap_decay} shows the decay in the minimal gap length for every $q$ in this range. We can observe that the values $p=1, q-1$ and $p=2, q-2$ are special and that give the upper and lower bounds on gap length for every odd $q$. The following statement has been checked to be \GapLengthConjecture \phantom{ }for $q \le \QmaxThouless$:

\begin{conjecture}
There is a constant $8\ge R> \CEYeightpq$ such that for any $p$ and $q$, being $q$ odd and $p$ and $q$ coprime, the length  $\gamma_k(p/q)$  of the gap with label $k$ of the periodic problem $\Sigma(p/q)$ satisfies
 \[
 \min_{k}\gamma_k(1/q)>\min_{p,k}\gamma_k(p/q) > \min_{k}\gamma_k(2/q)> R^{-k}.
 \]
 \end{conjecture}

Figure \ref{fig:thouless_scaling} shows the Lebesgue measure of the periodic spectrum as a function of the rational approximation $p/q$, which is known to tend to zero when $p/q$ tends to an irrational frequency $\omega$ . The rate at which this limit tends to zero is related to a conjecture by Thouless \cite{thoulessScalingDiscreteMathieu1990} namely, for any sequence of coprime $p_n$ and $q_n$ such that $p_n/q_n \to \omega$ irrational one has
\begin{equation}\label{eq:thoulessconjecture}
\lim_{n} q_n \cdot |\Sigma_{p_n/q_n}| = G:=\frac{32 C}{\pi}
\end{equation}
where $C=0.915\ldots$ is Catalan's constant, see \cite{helfferTotalBandwidthRational1995a, jitomirskaya2019critical}. Figure \ref{fig:thouless_scaling} displays the measure of the spectrum of $\Sigma_{p/q}$ as a function of the rational approximation $p/q$. 

\begin{figure}
	\centering
	\includegraphics[width=0.99\linewidth]{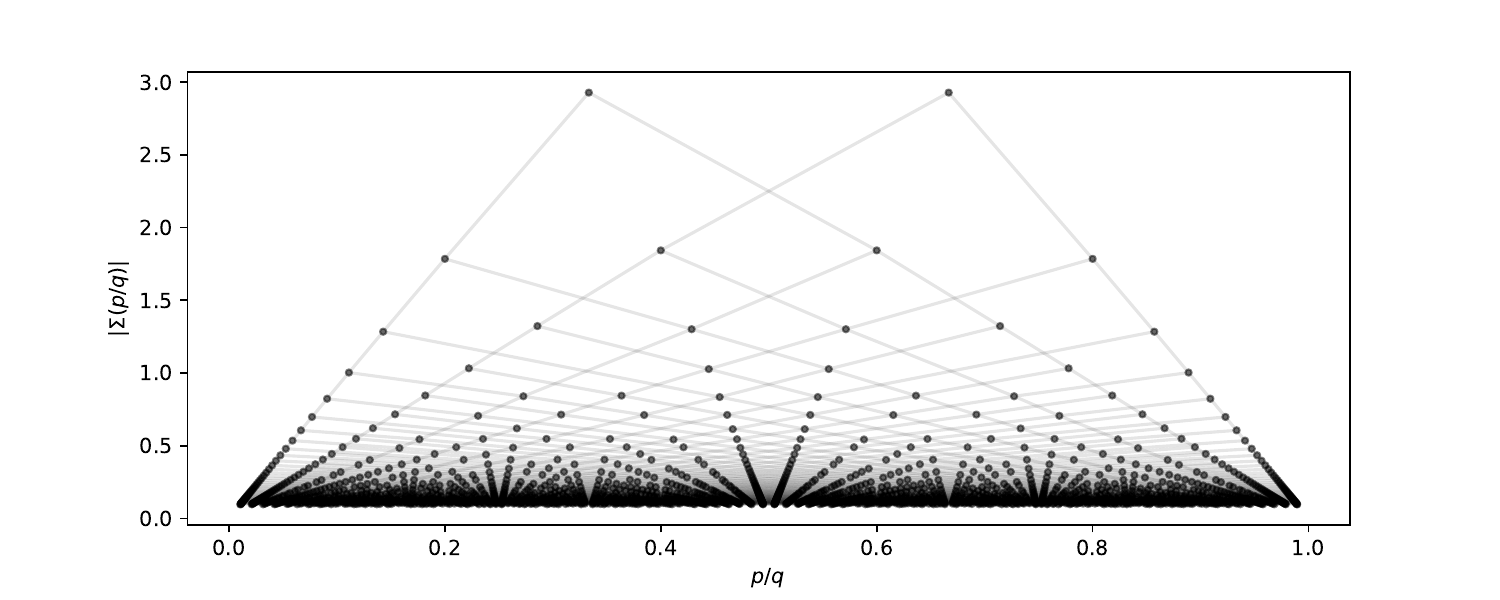}
	\includegraphics[width=0.99\linewidth]{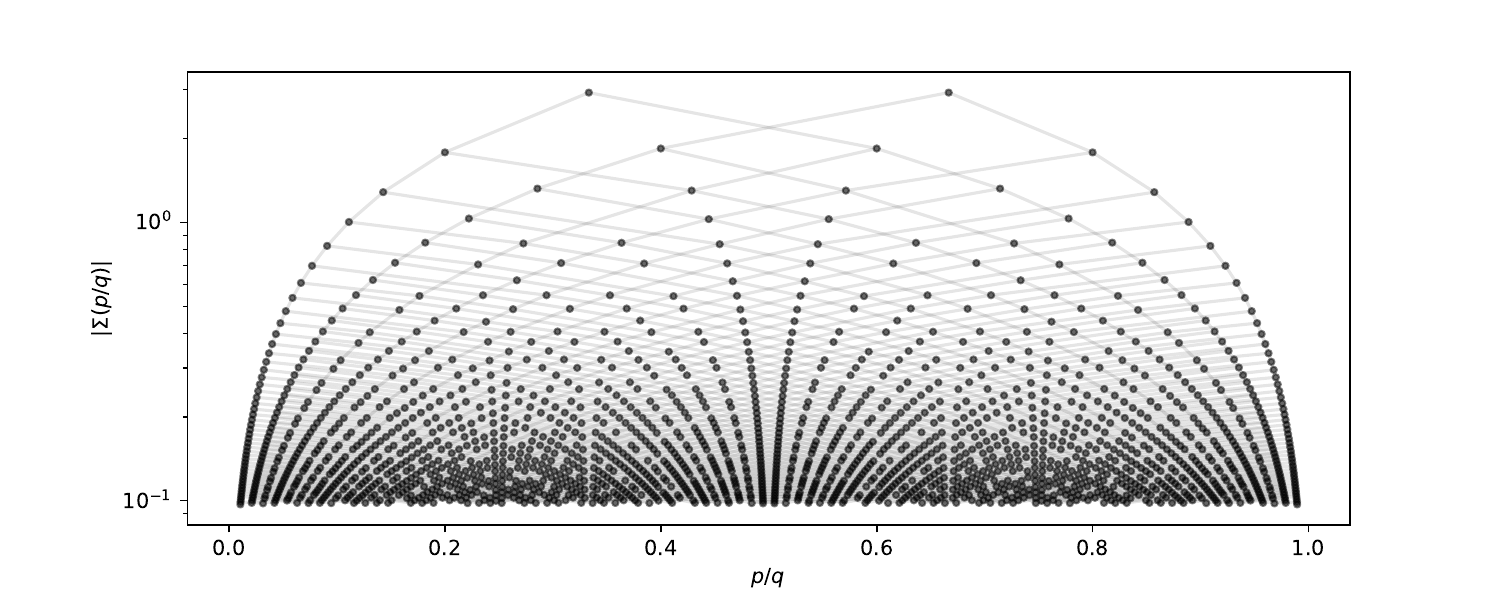}
	\caption{Measure of the AMO for rational frequencies as a function of the rational frequency $p/q$ for odd $q$ and $q\le \QmaxThouless$ (top) and the same with logarithmic y scale. Greylines connect dots whose frequency $p/q$ have the same $p$. It ressembles of Thomae's function, although the vertical values, the measures $\Sigma_{p,q}$ are not exactly the same for all values $p/q$ with the same $q$.}
	\label{fig:thouless_scaling}
\end{figure}

The convergence of the limit (\ref{eq:thoulessconjecture}) can also be seen in Figure \ref{fig:convergence_to_thouless} where the  value of the difference between  $32 C/\pi$ and   $q\cdot |\Sigma_{p/q}|$  is shown. In all the computed values, $32C/\pi$ was greater than $q\cdot |\Sigma_{p/q}|$. By checking the measures for each of the finite-dimesional matrices, we checked that the  following conjecture  was \ThoulessScalingConjecture  \phantom{ } for  $q\le \QmaxThouless$,

\begin {conjecture}
For any  $p$ and $q$  coprime, being $q$ odd, 
\[
\frac{32C}{\pi} > q\cdot |\Sigma_{p/q}|> q\cdot |\Sigma_{1/q}|>\sqrt{\frac{2}{q}}, p\ne 1, q-1. 
\]

\end{conjecture}
If such inequalities were shown to be true for every coprime $p$ and $q$  (which cannot be settled with the methods in the present paper), the analysis of this conjecture would be reduced to the consideration of the cases $p=1$ and $p=q-1$ \cite{helfferTotalBandwidthRational1995a}. See also the discussion in \cite{lastSumRuleDispersion1992a}.

\begin{figure}
\includegraphics[width=0.99\linewidth]{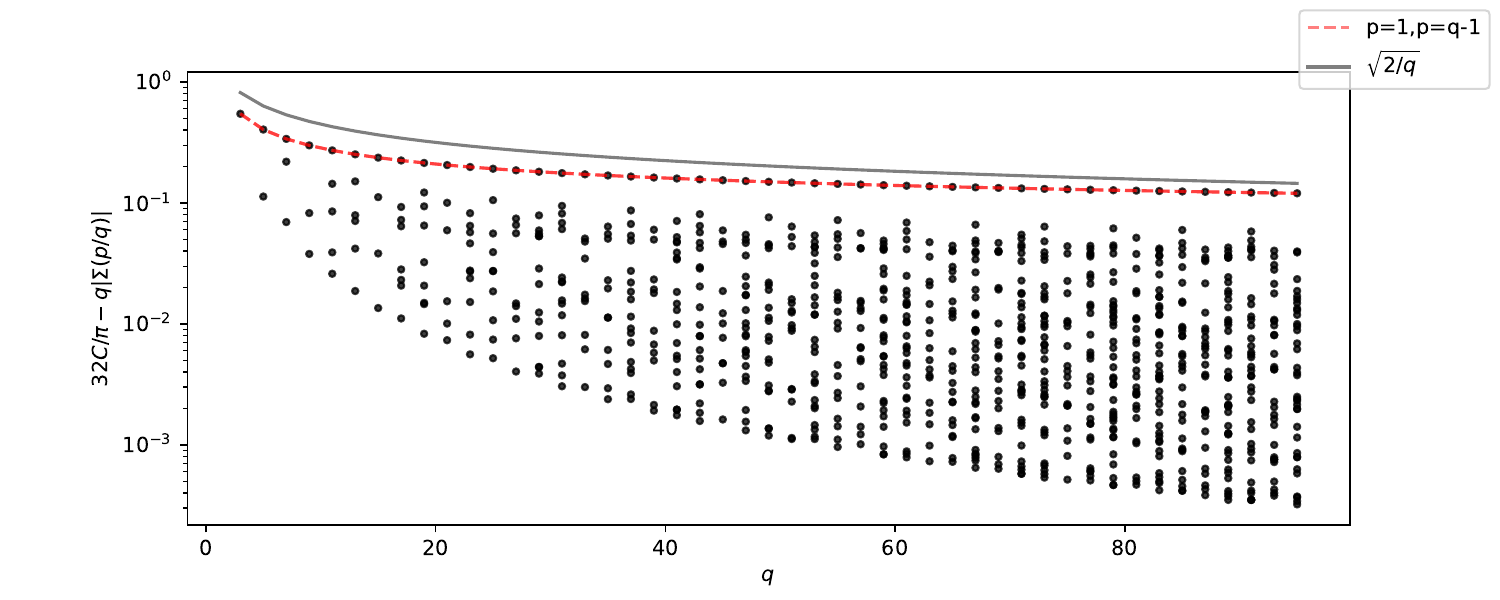}\\ 
\includegraphics[width=0.99\linewidth]{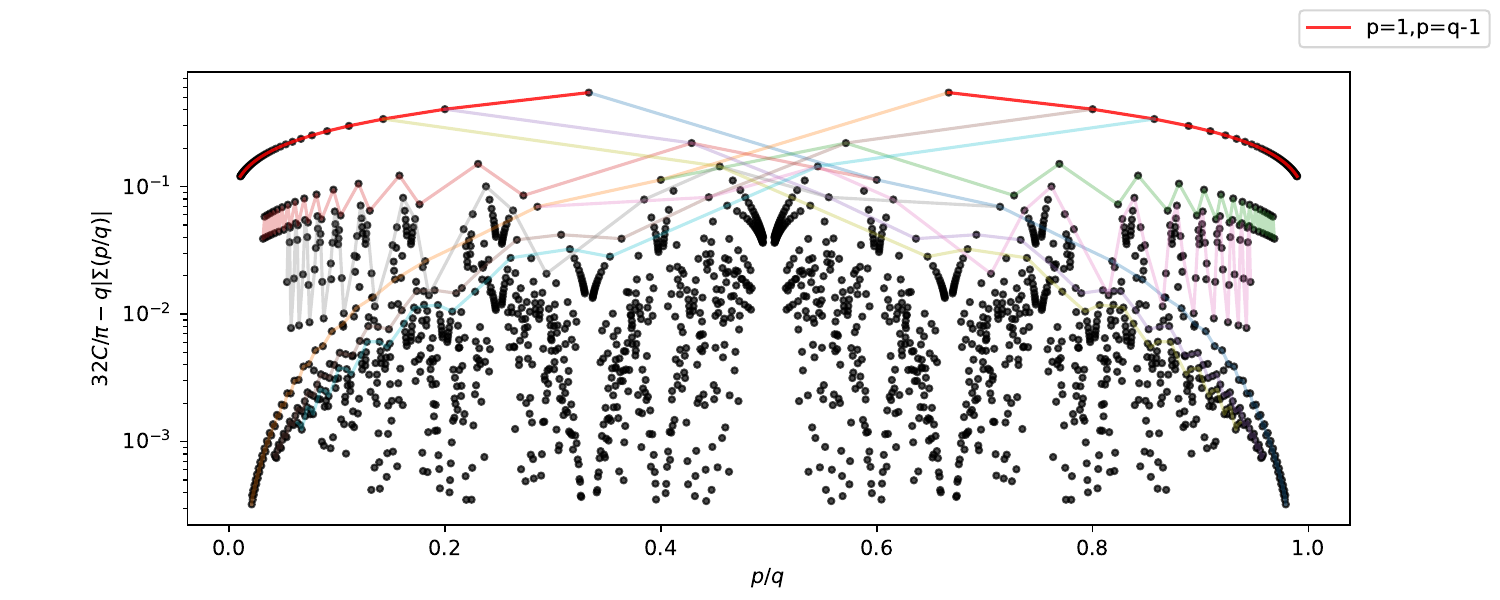}
\caption{Top: Difference $32C/\pi-q|\Sigma_{p/q}|$ for the as a function of $q$. Note that the slowest convergence is given exactly when $p=1$ or $p=q-1$, which seems to be faster than $\sqrt{2/q}$. Bottom: Same plot as a function of the rational approximations $p/q$. }
	\label{fig:convergence_to_thouless}
\end{figure}

Finally we wanted to investigate the Hölder continuity of the spectrum in Equation (\ref{eq:HolderK}) and obtain lower bounds given the exact computation of the spectrum for some rational $p/q$. We used the fact that, for even $q$, the spectrum of $\Sigma_{p/q}$ contains $0$, which is not in the spectrum when $q$ is odd. Figure \ref{fig:holder} displays, for every $p/q$ with $q$ the minimum positive value in the spectrum, which is precisely $d(0,\Sigma_{p/q})$ and allows to compute the constant in \ref{eq:HolderK}, proved to be lower that $60$ in \cite{Krasovsky2016}. For instance, computing the distance from $0$, which lies in the spectrum of $\Sigma_{1/2}$, with the rest of the spectrum we obtained that the constant needs to be larger than $\LowerBoundHolderConstant$.
\begin{figure}
    \centering
    \includegraphics[width=1\linewidth]{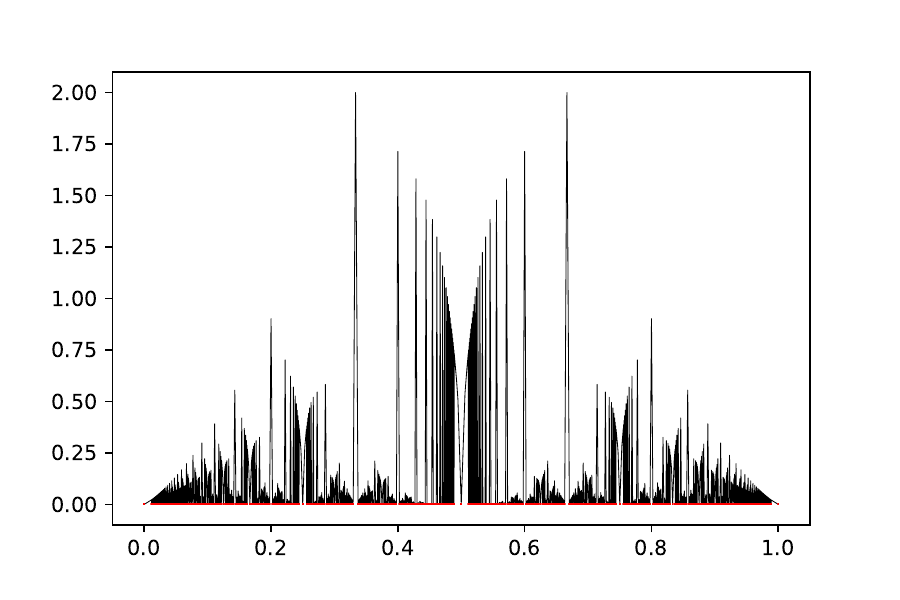}
    \caption{Plot of the minimum positive value in the spectrum $\Sigma_{p/q}$ for every odd $q \le \QmaxThouless$ as a function of the frequency $p/q$. Values in red correspond to points of the form $(j/k,0)$ with even $k$ and $gcd(j,k)=1$ which give a value of zero because $\{0\}$ is a collapsed gap for these frequencies with even $k$.}
    \label{fig:holder}
\end{figure}

\section{Discussion and conclusions}

In this paper we have examined two strategies for proving the existence of open gaps in the spectrum of the Almost Mathieu operator at critical coupling and some selected irrational frequencies. We have produced a number of gaps for selected frequencies, namely $\omega_1=(\sqrt{5}-1)/2$ and $\omega_2= e-2$.  Although the fact that these gaps were open was expected from previous numerical computations \cite{Hofstadter1976,papillon,Lamoureux2010,Satija2016}, the present paper shows that the use of validated numerics allows for rigorously determining this opening. Given the small differences in the eigenvalues of matrices involved, the use of extended precision and validated algorithms is convenient. Moreover, the lower bounds on the resolvent set seem to be quite realistic when compared to these    numerical values. The difference between the expected zero measure of the spectrum and the actual bound found with rigorous numerics, which is small but not zero, seems to be common in similar problems in dynamics. For instance, this has been observed in the AMO with $b=1$ and $3$ in \cite{Haro_Book_2016} and for circle maps in \cite{Figueras2020}.

Computer-assisted proofs  are increasingly used in all areas of mathematics and we hope that the constructive methods presented in this paper can be formalised with the use of proof assistants in the future.  

In this particular model, the Almost Mathieu operator, the spectral method seems to give slightly better results, although the overall performance, in terms of gaps and computing time, are similar. However, for more general models, we expect the dynamical method to perform better since Chambers formula, which reduces the computation of the periodic problem to simply two phases, $\theta_0=0$ or $\theta_{0}= 1/2q$, is specific of the Almost Mathieu operator and the spectral method would be much slower. The dynamical method can easily adapted to other sampling functions than the cosine. Extension to higher frequencies, moreover, is straightforward, although technically challenging by the implementation problems of higher-dimensional Fourier series. 

Although, to our knowledge, the opening of some specific gaps for the Almost Mathieu operator had not previously been shown, the results in the present paper give little clue on the opening of all gaps. Our main motivation in presenting them was twofold. On one hand, to show that it is possible to prove realistic bounds on the gaps of the Hofstadter Butterfly which account for most of what can be experimentally observed, both from a numerical point of view \cite{puig-simo:discrete} and from physical experiments \cite{Roushan2017b}. On the other, by being able to rigorously bound some gaps, one can test some of the inequalities and bounds on gaps, most notably the opening of all gaps, for some finite bound. In this paper, we have tested some of these conjectures like Thouless conjecture, the possible ordering of gaps according to their label (which we proved to be false in general) and the same could be done with analytic upper bounds for the spectrum such as in \cite{boca-zaharescu:norm}.

\section*{Acknowledgements}

The research of JP is supported by the grant PID-2021-122954NB-100 funded by MCIN/AEI (Spain) and “ERDF: A way of making Europe”.

\appendix

\section{Code and Data Availability}

The Python library to reproduce the results can be found in \cite{ValidatedHofstadter} and the computed data can be found in \cite{ValidatedDataset}.

\end{document}